\documentclass[12pt,leqno]{article}
\usepackage{amsmath,amssymb,mathrsfs,amsthm,bm,ddag}

\usepackage{xy}
\input{xy}
\xyoption{all}

\usepackage{threeparttable}

\usepackage{xspace}

\usepackage[geometry]{ifsym}

\makeatletter
\def\@seccntformat#1{\csname the#1\endcsname.\hspace{2ex}}

 \renewcommand{\subsection}%
  {\@startsection{subsection}%
  {2}%
  {\z@}%
  {2ex}
  {0ex}
  {\reset@font\normalsize\bfseries}}%

 \@addtoreset{equation}{subsection}

 \newcommand{\nsection}{\@startsection{section}{1}{\z@}%
     {-5ex}
     {1ex}
     {\reset@font\center\large\sc}}

 \renewenvironment{thebibliography}[1]
 {\nsection*{\refname\@mkboth{\refname}{\refname}}%
   \list{\@biblabel{\@arabic\c@enumiv}}%
   {\settowidth
   \labelwidth{\@biblabel{#1}}%
   \leftmargin
	\labelwidth
        \advance
	 \leftmargin
	 \labelsep
         \@openbib@code
         \usecounter{enumiv}%
         \let\p@enumiv\@empty
	 \parskip=0pt
	 \itemsep=1pt
	 \parsep=1pt
	 \itemindent=\z@
         \renewcommand\theenumiv{\@arabic\c@enumiv}}%
   	 \sloppy
   	 \clubpenalty4000
   	 \@clubpenalty\clubpenalty
   	 \widowpenalty4000%
   	 \footnotesize
   	 \sfcode`\.\@m}
  	 {\def\@noitemerr
    	 {\@latex@warning{Empty `thebibliography' environment}}%
   	 \endlist}

\makeatother

\newtheoremstyle{thm}
 {1em}
 {3pt}
 {\itshape}
 {}
 {\bf}
 {. ---}
 {0.5em}
 {}
\newtheoremstyle{dfn}
 {1em}
 {3pt}
 {}
 {}
 {\bf}
 {. {---}}
 {0.5em}
 {}

\swapnumbers
\theoremstyle{thm}
\newtheorem{thm}[subsection]{Theorem}
\newtheorem{lem}[subsection]{Lemma}
\newtheorem*{lem*}{Lemma}
\newtheorem{cor}[subsection]{Corollary}
\newtheorem*{cor*}{Corollary}
\newtheorem{prop}[subsection]{Proposition}
\newtheorem*{prop*}{Proposition}

\newtheorem*{conj*}{Conjecture}
\newtheorem*{thm*}{Theorem}

\theoremstyle{dfn}

\newtheorem*{dfn*}{Definition}

\newtheorem*{ex*}{Example}
\newtheorem{rem}[subsection]{Remark}
\newtheorem*{rem*}{Remark}

\newcommand{\shom}{\mathop{\mc{H}om}\nolimits}
\newcommand{\send}{\mathop{\mc{E}nd}\nolimits}

\usepackage{color}
\newenvironment{meta1}{
\noindent\color{red}
\sffamily[}{\upshape]}
\newenvironment{meta2}{
\noindent\color{magenta}
\sffamily[}{\upshape]}

\newsavebox{\circlebox}
\savebox{\circlebox}{\fontencoding{OMS}\selectfont\char13}
\newlength{\circleboxwdht}

\renewcommand{\H}{\ms{H}}

\begin{document}
\title{A Lefschetz theorem for overconvergent isocrystals with Frobenius
structure}

\author{
Tomoyuki Abe\footnote{
Kavli Institute for the Physics and Mathematics of the
Universe (WPI), The University of Tokyo, 5-1-5 Kashiwanoha, Kashiwa,
Chiba, 277-8583, Japan; supported by Grant-in-Aid for Young Scientists
(A) 16H05993.}
\,\,\, and \,\, 
H\'el\`ene Esnault\footnote{
Freie Universit\"at Berlin, Arnimallee 3, 14195, Berlin, Germany;
supported by the Einstein program.}
\setcounter{footnote}{-1}
}
\date{\today}
\maketitle

\begin{abstract}
 We show a Lefschetz theorem for irreducible overconvergent
 $F$-isocrystals on smooth varieties defined over a finite field. We derive
 several consequences from it.
\end{abstract}

\section*{Introduction} \label{intro}

Let $X_0$ be a normal geometrically connected scheme of finite type
defined over a finite field $\mathbb{F}_q$, let $\mathcal{F}_0$ be an
irreducible lisse Weil $\overline{ \mathbb{Q}}_\ell$-sheaf with finite
determinant (thus in fact $\mathcal{F}_0$ is \'{e}tale sheaf as well),
where $\ell \neq p={\rm char}(\mathbb{F}_q)$. In Weil II
\cite[Conj.~1.2.10]{weil2}, Deligne conjectured the following.
\begin{enumerate}
 \item[(i)] The sheaf $\mathcal{F}_0$ is of weight $0$.
 \item[(ii)] There is a number field
	     $E\subset\overline{\mathbb{Q}}_\ell$
	     such that for any $n>0$ and $x\in X_0(\mb{F}_{q^n})$,
	     the characteristic polynomial
	     $f_{x}(\mathcal{F}_0,t):=
	     {\rm det}(1-tF_{x}\mid
	     \mathcal{F}_{0,\overline{x}})\in
	     E[t]$, where $F_{x}$ is the geometric
	     Frobenius of $x$.
 \item[(iii)] For any $\ell' \neq p$ and any embedding $\sigma\colon
	      E\hookrightarrow \overline{\mathbb{Q}}_{\ell'}$,
	      for any $n>0$ and $x\in X_0(\mb{F}_{q^n})$,
	      any root of $\sigma f_x(\mc{F}_0,t)=0$ is an $\ell'$-adic
	      unit.
 \item[(iv)] For any $\sigma$ as in (iii), there is an irreducible
	     $\overline{ \mathbb{Q}}_{\ell'}$-lisse  sheaf
	     $\mathcal{F}_{0,\sigma}$, called the $\sigma$-companion,
	     such that $\sigma
	     f_{x}(\mathcal{F}_0,t)=f_{x}(\mathcal{F}_{0,\sigma},t)$.
 \item[(v)] There is a crystalline version of (iv).
\end{enumerate}
\medskip

Deligne's conjectures (i)--(iv) have been proved by Lafforgue
\cite[Thm.~VII.6]{Laf} when $X_0$ is a smooth curve, as a corollary of
the Langlands correspondence, which is proven showing that automorphic
forms are in some sense motivic.
\medskip

When $X_0$ has dimension at least $2$, the automorphic side on which one
could rely to prove Deligne's conjectures is not available: there is no
theory of automorphic forms in higher dimension.
The problem then becomes how to reduce, by geometry, the statements to
dimension $1$. For (i) and (iii), one proves a Lefschetz theorem (see
\cite[Thm.~2.5]{Dr}, \cite[1.5--1.9]{De},  \cite[B1]{EK}):

\begin{thm} \label{thm:lef_l}
On $X_0$ smooth,  for any closed point $x_0$, there a smooth curve $C_0$ and a
 morphism $C_0\rightarrow X_0$ such that $x_0\to X_0$ lifts
 to $x_0\to C_0$, and such that the restriction of $\mathcal{F}_0$ to
 $C_0$ remains irreducible.
\end{thm}
Using Theorem~\ref{thm:lef_l}, Deligne proved (ii) (\cite[Thm.~3.1]{De}),
and Drinfeld, using (ii), proved (iv) in (\cite[Thm.~1.1]{Dr}), assuming
in addition $X_0$ to be smooth. In particular Drinfeld proved in
\cite[Thm.~2.5]{Dr} the following key theorem.

\begin{thm} \label{thm:drin_tame}
 If $X_0$ is smooth, given a number field $E \subset
 \overline{\mathbb{Q}}_\ell$, and a place $\lambda$ dividing $\ell$, a
 collection of polynomials $f_{x}(t)\in E[t]$ indexed by any $n>0$ and
 $x\in X_0(\mb{F}_{q^n})$, such that the following two conditions are
 satisfied:
 \begin{itemize}
  \item[(i)] for any smooth curve $C_0$ with a morphism $C_0\to X_0$ and
	     any $n>0$ and $x\in C_0(\mb{F}_{q^n})$,
	     there exists a lisse \'{e}tale $\overline{
	     \mathbb{Q}}_\ell$-sheaf $\mathcal{F}_0^{C_0}$ on $C_0$ with
	     monodromy in $\mr{GL}(r,E_\lambda)$ such that
	     $f_{x}(t)=f_{x}(\mathcal{F}_0^{C_0},t)$; 
  \item[(ii)] there exists a finite \'etale cover $X'_0\to X_0$ such
	      that $\mathcal{F}_0^{C_0}$ is tame on all $C_0$ factoring
	      through $X'_0\rightarrow X_0$.
 \end{itemize} 
 Then there exists a lisse $\overline{\mathbb{Q}}_\ell$-sheaf
 $\mathcal{F}_0$ on $X_0$ with monodromy in $\mr{GL}(r, E_\lambda)$,
 such that for any $n>0$ and $x\in X_0(\mb{F}_{q^n})$,
 $f_{x}(t)=f_{x}(\mathcal{F}_0,t)\in E[t]$.
\end{thm}
Further,  to realize the assumptions of Theorem~\ref{thm:drin_tame} in
order to show the existence of  $\mathcal{F}_{0,\sigma}$, Drinfeld uses
Theorem~\ref{thm:lef_l}  in  \cite[4.1]{Dr}. He constructs step by step
the residual representations with monodromy in $\mr{GL}(r,
E_\lambda/\frak{m}^n)$ for $n$ growing, where $\frak{m}\subset
E_\lambda$ is the maximal ideal.

\medskip

The formulation of (v) has been made explicit by Crew
\cite[4.13]{Cre92}. The conjecture is that the crystalline category
analogous to the category of Weil  $\overline{ \mathbb{Q}}_\ell$-sheaves
is the category of overconvergent $F$-isocrystals (see
Section~\ref{recallisoc} for the definitions). 
In order to emphasize  the analogy between $\ell$ and $p$, one slightly
reformulates the definition of companions. One replaces $\sigma$ in
(iii) by an isomorphism $\sigma\colon \overline{ \mathbb{Q}}_\ell \to
\overline{\mathbb{Q}}_{\ell'}$ (see \cite[Thm.~4.4]{EK}), and keeps
(iv) as it is.  Here $\ell$, $\ell'$ are any two prime numbers.
For $\ell'=p$, $\ell\neq p$, and $\mc{F}$ an irreducible lisse
$\overline{\mathbb{Q}}_\ell$-sheaf, one requests the existence of an
overconvergent $F$-isocrystal $M_0$ on $X_0$ with eigenpolynomial
$f_{x}(M_0,t)$ such that  $f_{x}(M_0,t)=\sigma
f_{x}(\mathcal{F},t)\in \sigma(E)[t]$ for any $n>0$ and $x\in
X_0(\mb{F}_{q^n})$, where $f_{x}(\mathcal{F},t)$ is the characteristic
polynomial of the geometric Frobenius at $x$ on $\mc{F}$ (see
Section~\ref{eigenpol} for the definitions). The isocrystal $M_0$ is
called a {\em $\sigma$-companion to $\mathcal{F}$}.
Given an irreducible overconvergent $F$-isocrystal $M_0$
with finite determinant on $X_0$, and $\sigma$ as above, a lisse
$\ell$-adic Weil sheaf $\mathcal{F}$ on $X_0$ is a
$\sigma^{-1}$-companion if $\sigma^{-1} f_{x}(M_0,t)=
f_{x}(\mathcal{F},t)\in E[t]$ at $x\in X_0(\mb{F}_{q^n})$
(see Definition~\ref{eigenpol}).
Similarly we can assume $p=\ell=\ell'$. This way we can talk on
$\ell$-adic or $p$-adic companions of either an $M_0$ or an
$\mathcal{F}$.
The companion correspondence should
preserve the notions of
irreducibility, finiteness of the determinant,  the eigenpolynomials at
closed points of $X_0$, and the ramification.
\medskip

The conjecture in the strong form  has been proven by the
first author when $X_0$ is a smooth curve (\cite[Intro.\ Thm.]{A}).
The aim of this article is to prove the following analog of
Theorem~\ref{thm:lef_l} on $X$ smooth.

\begin{thm}[Theorem~\ref{thm:lefschetz}] \label{thm:lef_p}
 Let $X_0$ be a smooth geometrically connected scheme over $k$.
 Let $M_0$ be an irreducible overconvergent $F$-isocrystal with finite
 determinant. Then there exists a dense open subscheme
 $U_0\hookrightarrow X_0$, such that for every closed point $x_0\to
 U_0$, there exists a smooth irreducible curve $C_0$ defined over $k$,
 together with a morphism $C_0\to X_0$ and a factorization $x_0\to
 C_0\to X_0$,  such that
 the pull-back of $M_0$ to $C_0$ is irreducible.
\end{thm}

Theorem~\ref{thm:lef_p}, together with \cite[Rmk.~3.10]{De}, footnote 2,
and \cite[Thm.~4.2.2]{A}  enable one to conclude that there is a number
field $E\subset \overline{\mathbb{Q}}_p$ such that for any $n>0$ and
 $x\in X_0(\mb{F}_{q^n})$, $f_{x}(M_0,t)\in E[t]$ (see
 Lemma~\ref{lem:E}). This yields the $p$-adic analog of (i).
Then  Theorem~\ref{thm:drin_tame} implies  the existence of
$\ell$-adic companions to a given irreducible overconvergent
$F$-isocrystal $M_0$ with finite determinant (see
Theorem~\ref{thm:comp}). We point out that the existence of $\ell$-adic
companions has already been proven by Kedlaya in \cite[Thm.~5.3]{K2} in
a different way, using weights (see \cite[\S 4, Intro.]{K2}), however
not their irreducibility.
The Lefschetz theorem~\ref{thm:lefschetz} implies that the companion
correspondence preserves irreducibility.
\medskip

Theorem~\ref{thm:lef_p} has other consequences (see
Section~\ref{sec:appl}), aside of the existence already mentioned of $\ell$-adic companions. 
Deligne's finiteness theorem \cite[Thm.~1.1]{EK} transposes to
the crystalline side (see Remark~\ref{cor:finiteness}): on $X_0$ smooth,
there are finitely many isomorphism classes of irreducible
overconvergent $F$-isocrystals in bounded rank and bounded ramification,
up to twist by a character of the finite field. One can also kill the
ramification of an $F$-overconvergent isocrystal by a finite \'etale
cover in Kedlaya's semistability reduction theorem
(Remark~\ref{rmk:ss}).
\medskip

We now explain the method of proof of Theorem~\ref{thm:lef_p}. We
replace $M_0$ by the full Tannakian subcategory $\langle M\rangle $
of the category of overconvergent $F$-isocrystals spanned by $M$ over the algebraic closure
$\overline{ \mathbb{F}}_q$ (we drop the lower indices $_0$ to indicate
this, see Section~\ref{recallisoc}  for the definitions).  We slightly improve 
the theorem (\cite[Prop.~2.21 (a), Rmk.~2.29]{DM})
describing the surjectivity of an homomorphism of Tannaka groups in
categorical terms in Lemma~\ref{keytannak}:
the restriction functor $\langle
M\rangle \to \langle M|_C\rangle$ to a curve $C\to X$ is an isomorphism
as it is fully faithful and any  $F$-overconvergent isocrystal of rank
$1$ on $C$ is torsion.
Class field theory for $F$-overconvergent isocrystals
(\cite[Lem.~6.1]{AbeLL}) implies the torsion property.  As for full
faithfulness, the problem is of cohomological nature, one has
to compute that the restriction homomorphism
$H^0(X, N)\to H^0(C, N|_C)$ is an isomorphism for all objects in
$\langle M\rangle$. In the tame case, this is performed in
Section~\ref{sec:coh}  using the techniques developed in \cite{AC}. As a
corollary, $\ell$-adic companions exist in the tame case (see
Proposition~\ref{comptame}). In the wild case, Kedlaya's semistability
reduction theorem asserts the existence of a good alteration
$h\colon X''_0\to X_0$ such that  $h^+M_0$ becomes tame.
One considers the $\ell$-adic companion $\mathcal{F}_0$ of $h^+M_0$ and
a finite \'etale cover $g\colon X'_0\to X''_0$  which is such that
curves $C_0\to X_0$ with non-disconnected pull-back  $C_0\times_{X_0}
X'_0$  have the property that $C_0\times_{X_0} X''_0$   preserves the
irreducible constituents of $\mathcal{F}_0$ (see
Lemma~\ref{lem:curve}). It remains then to show that the  dimensions of
$H^0$ of $M$ and of $\mathcal{F}$,  pulled back  $X'_0$, $\mathcal{F}_0$, are
the same (see Lemma~\ref{coinH0}).
\medskip

{\it Acknowledgments:}
The first author thanks Kiran S. Kedlaya for discussions before this
work took shape. He is also grateful to Valentina Di Proietto for
discussions, encouragements, and support while he was visiting
Freie University.
The second author thanks Moritz Kerz for the discussions while writing
\cite{EK}, which enabled her to better understand Deligne's ideas,
Atsushi Shiho with whom she discussed various versions of Lefschetz
theorems for unit-root isocrystals, and Pierre Deligne who promptly
answered a question concerning \cite[Rmk.~3.10]{De} (see footnote 2).

\section*{Notations and conventions}
Let $q=p^s$, and let $k$ be a field of $q$ elements $k$. We fix once for
all an algebraic closure $\overline{k}$ of $k$.
For an integer $n>0$, let $k_n$ be the finite extension of in $\overline{k}$
of degree $n+1$.
For every prime number $\ell$, including $p$, we fix an isomorphism
$\iota\colon\overline{\mb{Q}}_p\cong\mb{C}$. 
By a curve we mean  an irreducible
scheme of finite type over $k$ which is of dimension $1$.

\section{Generalities}
\subsection{}
\label{recallisoc}
Let us start with recalling basic concepts of $p$-adic coefficients used
in \cite{A}. We fix a base tuple (cf.\ \cite[4.2.1]{A})
$\mf{T}_k:=(k,R=W(k),K,s,\mr{id})$, and take
$\sigma_{\overline{\mb{Q}}_p}=\mr{id}$.
Let $X_0$ be a smooth scheme over $k$.
We put $X_n:=X_0\otimes_kk_n$, and $X:=X_0\otimes_k\overline{k}$.
We denote the categories
$\mr{Isoc}^\dag_{\mf{T}_k}(X_0/\overline{\mb{Q}}_p)$ and 
$\mr{Isoc}^\dag_{\mf{T}_k}((X_0)_0/\overline{\mb{Q}}_p)$,
which are defined in \cite[4.2.1 {\it etc.}]{A},
by $\mr{Isoc}^\dag(X)$ and $\mr{Isoc}^\dag(X_0)$
respectively. We recall now the definitions.
\medskip

In \cite{Be}, Berthelot defined the category of overconvergent
isocrystals,
which we denote by $\mr{Isoc}^\dag_{\mr{Ber}}(X_0/K)$.
We fix an algebraic closure $\overline{\mb{Q}}_p$ of $K$ and  extend the
scalars from $K$ to $\overline{\mb{Q}}_p$  in the following way.
Let $L$ be a finite field extension of $K$ in
$\overline{\mb{Q}}_p$. Then an $L$-isocrystal is a pair $(M,\lambda)$
where $M\in\mr{Isoc}^\dag_{\mr{Ber}}(X_0/K)$, and
$\lambda\colon L\rightarrow\mr{End}_{\mr{Isoc}^\dag
_{\mr{Ber}}(X_0/K)}(M)$ is a ring homomorphism which is called the {\em
$L$-structure}.
Homomorphisms between $L$-isocrystals are homomorphisms of isocrystals
which are compatible with the $L$-structure.
The category of $L$-isocrystals is denoted by
$\mr{Isoc}^\dag_{\mr{Ber}}(X_0/L/K)$. Finally, taking the
2-inductive limit over all such $L$, we obtain
$\mr{Isoc}^\dag_{\mr{Ber}}(X_0/\overline{\mb{Q}}_p/K)$. This category
does not have suitable finiteness properties. To be able to acquire
these, we need, in addition, to define a ``Frobenius structure''.
\medskip

Let $F\colon X_0\rightarrow X_0$ be the $s$-th Frobenius endomorphism of
$X_0$, which is an endomorphism over $k$.
For an integer $n>0$
and $M\in\mr{Isoc}^\dag_{\mr{Ber}}(X_0/\overline{\mb{Q}}_p/K)$, an
{\em $n$-th Frobenius structure}\footnote{
In the definition, the functor $F^+$ is used. This functor is the same
as the more familiar notation $F^*$ in \cite{Be} (cf.\
\cite[Rem~1.1.3]{A}).
Since our treatment of isocrystals is from the viewpoint of
$\ms{D}$-modules, we borrow the notations from this theory.}
is an isomorphism
$\Phi\colon F^{n+}M\xrightarrow{\sim}M$.
The category $\mr{Isoc}^\dag_{\mf{T}_k}(X_0)$ is the category of
pairs $(M,\Phi)$ where
$M\in\mr{Isoc}^\dag_{\mr{Ber}}(X/\overline{\mb{Q}}_p/K)$, $\Phi$ is
the $1$-st Frobenius structure,
and the homomorphisms are the ones compatible with $\Phi$ in an obvious
manner.
We recall the following result.

\begin{lem*}[{\cite[Cor.~1.4.12]{A}}]
 \label{lem:k}
 Let $k'$ be a finite extension of $k$, and $X_0$ be a smooth scheme
 over $k'$, then we have a canonical equivalence
 \begin{equation*}
  \mr{Isoc}^\dag_{\mf{T}_k}(X_0)\cong
   \mr{Isoc}^\dag_{\mf{T}_{k'}}(X_0).
 \end{equation*}
\end{lem*}
This lemma shows that the category does not depend on the choice of the
base field, and justifies the notation $\mr{Isoc}^\dag(X_0)$ by
suppressing the subscript $\mf{T}_k$.
\medskip

Finally, $\mr{Isoc}^\dag(X)$ is the full subcategory of
$\mr{Isoc}^\dag_{\mr{Ber}}(X_0/\overline{\mb{Q}}_p/K)$
consisting of objects $M$ such that for any constituent $N$ of $M$,
there exists $i>0$ such that $N$ can be endowed with $i$-th Frobenius
structure.
This category {\em does} depend on the choice of the base.
\medskip

As a convention, we put subscripts $\cdot_n$ for isocrystals on
$X_n$. Let $M_n\in\mr{Isoc}(X_n)$. Then the pull-back of $M_n$ to
$X_{n'}$ for $n'\geq n$ is denoted by $M_{n'}$, and pull-back to
$\mr{Isoc}^\dag(X)$ is denoted by $M$.

\begin{rem*}
 The category $\mr{Isoc}^\dag(X_n)$ has another description by
 \cite[Lem.~1.4.12 (ii)]{A}: it is equivalent to the category of
 isocrystals with $n$-th Frobenius structure 
  in $\mr{Isoc}^\dag_{\mr{Ber}}(X_0/\overline{\mb{Q}}_p/K)$.
\end{rem*}

\subsection{}
\label{dfntame}
For later use, we fix the terminology of tameness.

\begin{dfn*}
 Let $X_0$ be a smooth curve, and  $\overline{X}_0$ be the smooth
 compactification of $X_0$. An isocrystal $M\in\mr{Isoc}^\dag(X)$ is
 said to be {\em tame} if it is log-extendable along the boundary
 $\overline{X}_0\setminus X_0$ \cite[\S1]{S0}.
 For a general scheme $X_0$ of finite type
 over $k$, $M\in\mr{Isoc}^\dag(X)$ is said to be {\em tame} if for any
 smooth curve $C_0$ and any morphism $\varphi\colon C_0\rightarrow X_0$,
 the pull-back $\varphi^+(M)$ is tame. We say $M\in\mr{Isoc}^\dag(X_0)$
 is {\em tame} if $M$ is tame.
\end{dfn*}

\begin{rem*}
 Let $X_0$ be a smooth scheme which admits a smooth compactification
 whose divisor at infinity has strict normal crossings. Then
 $M\in\mr{Isoc}^\dag(X)$ being tame is equivalent to saying $M$ is
 log-extendable along the divisor at infinity. The ``if'' part is easy
 to check, and the ``only if'' part is a consequence of \cite[Thm.~0.1]{S}.
 We need to be careful as Shiho is assuming
 the base field to be uncountable. However, in our situation, because of
 the presence of Frobenius structure, this assumption is not needed as
 explained in \cite[footnote (4) of 2.4.13]{A}. 
 \end{rem*}

\subsection{}
\label{arithdmodrev}
In these notes, we freely use the formalism of arithmetic
$\ms{D}$-modules developed in \cite{A}.
For a separated scheme of finite type $X_0$ over $k$, in {\it ibid.}, 
the triangulated
category $D^{\mr{b}}_{\mr{hol},\mf{T}_k}(X)$ (resp.\
$D^{\mr{b}}_{\mr{hol},\mf{T}_k}(X_0)$) with t-structure is defined. Its
heart is denoted by $\mr{Hol}_{\mf{T}_k}(X)$ (resp.\
$\mr{Hol}_{\mf{T}_k}(X_0)$), see \cite[1.1]{A}.
The cohomology functor for this t-structure is denoted by $\H^*$.
We often drop the subscript $\mf{T}_k$.
When $X$ is smooth, $\mr{Isoc}^\dag(X)$ (resp.\ $\mr{Isoc}^\dag(X_0)$)
is fully faithfully embedded
into $\mr{Hol}(X)[-d]$ (resp.\ $\mr{Hol}(X_0)[-d]$) where $d=\dim(X_0)$,
and we identify $\mr{Isoc}^\dag(X)$ (resp.\ $\mr{Isoc}^\dag(X_0)$) with
its essential image in $\mr{Hol}(X)[-d]$ (resp.\ $\mr{Hol}(X_0)[-d]$).
Let $\epsilon \colon X_0\rightarrow\mr{Spec}(k)$ be the structural
morphism. For $M\in\mr{Isoc}^\dag(X)$, we set
\begin{equation*}
 H^i(X/k,M):=H^i(\epsilon_+M)
\end{equation*}
which is a $\overline{\mb{Q}}_p$-vector space.
If no confusion can arise, we often drop $/k$, even if  this cohomology
{\em does} depend on $k$, as the following lemma shows.

\begin{lem}
 \label{compdiffbase}
 Let $X_0$ be a variety over a finite field extension $k'$ of $k$.
 Let $M\in D^{\mr{b}}_{\mr{hol}}(X)$. Then one has
 \begin{equation*}
  H^*(X/k,M)\cong\bigoplus_{\sigma\in\mr{Gal}(k'/k)}
   H^*(X/k',M).
 \end{equation*}
\end{lem}
\begin{proof}
 By \cite[Cor.~1.4.12]{A}, we  have an equivalence
 $D^{\mr{b}}_{\mr{hol},\mf{T}_k}(X)\cong
 D^{\mr{b}}_{\mr{hol},\mf{T}_{k'}}(X)$, which is compatible with
 push-forwards. Thus, it reduces to showing the lemma for
 $X_0=\mr{Spec}(k')$, in which  case it is easy to check.
\end{proof}

\subsection{}
\label{eigenpol}
Let $X_0$ be a smooth scheme over $k$. Let $x\in X_0(k')$ for some
finite extension $k'$ of $k$, and $i_x\colon
\mr{Spec}(k')\rightarrow X_0$ be the corresponding morphism. We have
the pull-back functor $i_x^+\colon
\mr{Isoc}(X_0)\rightarrow\mr{Isoc}^\dag(\mr{Spec}(k'))$. The category
$\mr{Isoc}^\dag(\mr{Spec}(k'))$ is equivalent to the category of finite
dimensional $\overline{\mb{Q}}_p$-vector spaces $V$ endowed with an
automorphism  $\Phi$ of $V$.
Given $M_0\in\mr{Isoc}^\dag(X_0)$, we denote by $f_x(M_0,t) \in
\overline{\mb{Q}}_p[t]$ the eigenpolynomial of the automorphism $\Phi_x$
of $i^+_x(M_0)$, namely
\begin{equation*}
 f_x(M_0,t)=\det\bigl(1-t\Phi_x\mid i_x^+(M_0)\bigr)
  \in\overline{\mb{Q}}_p[t].
\end{equation*}
Similarly, for a lisse Weil $\overline{\mb{Q}}_\ell$-sheaf
${}_{\ell}M_0$ on $X_0$, we denote the characteristic polynomial
of the geometric Frobenius $F_x$ at $x$ by $f_x({}_{\ell}M_0,t)$:
\begin{equation*}
 f_x({}_{\ell}M_0,t)=\det\bigl(1-tF_x\mid
  {}_{\ell}M_{0,\overline{x}}\bigr)
  \in\overline{\mb{Q}}_{\ell}[t],
\end{equation*}
where $\overline{x}$ is a $\overline k$-point above $x$.

\begin{dfn*}
 Let $M_0\in\mr{Isoc}^\dag(X_0)$.
 \begin{enumerate}
  \item Let $\iota\colon\overline{\mb{Q}}_p\rightarrow\mb{C}$ be a field
	isomorphism.
	The isocrystal $M_0$ is said to be {\em $\iota$-pure of weight
	$w\in\mb{C}$} if the absolute value of any root of
	$f_x(M_0,t)\in\overline{\mb{Q}}_p[t]
	\xrightarrow[\iota]{\sim}\mb{C}[t]$ equals
	to $q_x^{w/2}$ for any $x\in X_0(k')$ with residue field of
	cardinality $q_x$.
	The isocrystal $M_0$ is said to be {\em $\iota$-pure} if it is
	$\iota$-pure of weight $w$ for some $w$.
	
  \item The isocrystal $M_0$ is said to be {\em algebraic} if
	$f_x(M_0,t) \in \overline{\mb{Q}}[t] \subset
	\overline{\mb{Q}}_p[t]$ for any $x\in X_0(k')$.
	    
  \item Given an isomorphism $\sigma\colon\overline{\mb{Q}}_p
	\rightarrow\overline{\mb{Q}}_\ell$ for a prime $\ell \neq
	p$, the isocrystal $M_0$ is said to be {\em $\sigma$-unit-root}
	if any root of $\sigma f_x(M_0,t)$ is an $\ell$-adic unit for
	any $x\in X_0(k')$.
       
  \item Given an automorphism $\sigma$ of $ \overline{\mb{Q}}_p$, a
	{\em $\sigma$-companion to  $M_0$} is an
	$M_0^\sigma \in\mr{Isoc}^\dag(X_0)$ such that $\sigma
	f_x(M_0,t) =f_x(M_0^\sigma, t)$ 
	for any $x\in X_0(k')$. One says that $M^\sigma_0$ is a
	{\em $p$-companion of $M_0$}.
	     
  \item Given an isomorphism $\sigma\colon\overline{\mb{Q}}_p
	\rightarrow\overline{\mb{Q}}_\ell$ for a prime $\ell \neq
	p$, a {\em $\sigma$-companion of $M_0$} is a lisse Weil
	$\overline{\mb{Q}}_\ell$-sheaf
	${}_{\ell}M_0^\sigma$ on $X_0$ such that $\sigma
	f_x(M_0,t)=f_x({}_{\ell}M^\sigma_0,t)$ for any $x\in X_0(k')$.
	One says that ${}_{\ell}M_0^\sigma$ is an {\em $\ell$-companion
	of $M_0$}. Abusing notations, we also write ${}_\ell M_0$ for an
	$\ell$-companion.
 \end{enumerate}
\end{dfn*}

\begin{rem*}
 Note that the eigenpolynomial is independent of the choice of the base
 field $k$. Therefore, the notion of algebraicity and $\ell$-adic
 companions are ``absolute'', they do not depend on the base field.
\end{rem*}

\subsection{}
\label{Delthmrec}
We recall the following theorem by Deligne.

\begin{thm*}[{\cite[Prop.~1.9 + Rmk.~3.10]{De}}]
 Let $X_0$ be a connected scheme of finite type over $k$. Assume given a
 function $t_n\colon
 X_0(k_n)\rightarrow\overline{\mb{Q}}_{\ell}[t]$ such that
 \begin{quote}
  {\normalfont (*)} for any  morphism $\varphi\colon
  C_0\rightarrow X_0$ from a smooth curve $C_0$, there exists a lisse
  Weil\footnote{In {\it ibid.}, Deligne assumes 
  the $\overline{\mb{Q}}_\ell$-sheaf to be lisse. However, in an
  email to the authors, he pointed out that it is enough to assume
  the  $\overline{\mb{Q}}_\ell$-sheaf to be a Weil sheaf, without
  changing his proof.}
  $\overline{\mb{Q}}_\ell$-sheaf ${}_{\ell}M[\varphi]$ on $C_0$
  such that for any $n$ and $x\in C_0(k_n)$, we have
  \begin{equation*}
   f_x({}_{\ell}M[\varphi],t)=t_n(\varphi(x)).
  \end{equation*}
 \end{quote}

\begin{itemize}
 \item[(i)] Assume there exists $x\in X_0(k_n)$ such that
	    $t_n(x)\in\overline{\mb{Q}}[t] \subset
	    \overline{\mb{Q}}_\ell[t]$ (resp.\ any root of $t_n(x)=0$ is
	    $\ell$-adic unit).
	    Then $t_n(x)\in\overline{\mb{Q}}[t] \subset
	    \overline{\mb{Q}}_\ell[t]$ for any $x\in X_0(k_n)$ (resp.\
	    any root of $t_n(x)=0$ is $\ell$-adic unit for any $x\in
	    X_0(k_n)$).
	   
 \item[(ii)] Assume that there exists a finite  \'{e}tale cover
	     $X'\rightarrow X$ such that for any $\varphi$ as in
	     {\normalfont (*)} above, the pull-back of ${}_\ell
	     M[\varphi]$ to $C' =X'\times_X C_0$ is
	     tamely ramified. If in addition, the assumption of (i) holds, then there
	     exists a number field $E$ in
	     $\overline{\mathbb{Q}}_\ell$ such that $t_n$
	     takes value in $E[t]$ for any $n$.
\end{itemize}
\end{thm*}

\begin{rem*}
 In \cite[Prop.~1.9]{De}, the assertion is formulated slightly
 differently: if there exists one closed point
 $x\in X_0(k_n)$ such that any root of $t_n(x)$ is a Weil
 number of weight $0$, then the same property holds for any points of $X_0$.
 The argument used by Deligne shows that if $t_n(x)$ is algebraic
 (resp.\ any root of $t_n(x)=0$ is $\ell$-adic unit) at one closed
 point $x$, then it is algebraic (resp.\ any root is $\ell$-adic unit)
 at all closed points.
\end{rem*}

\subsection{}
\label{keytannak}
We shall use the  following lemma on Tannakian categories to  show our
Lefschetz  theorem~\ref{thm:lefschetz}. 

\begin{lem*}
 Let $\Phi \colon (\mc{T}, \omega=\omega' \circ \Phi) \rightarrow
 (\mc{T}', \omega')$ be a tensor functor between neutral  Tannakian
 categories over $\overline{\mb{Q}}_p$ (or any field of characteristic
 $0$). If
 \begin{quote} 
  $(\star )$\  for any rank $1$ object $L\in\mc{T}'$, there exists an
  integer $m>0$ such that $L^{\otimes m}$ is in the essential image of
  $\Phi$,
 \end{quote}
 then the induced functor
 $\Phi^*\colon\pi_1(\mc{T}', \omega')\rightarrow\pi_1(\mc{T},\omega)$ is
 faithfully flat if and only if $\Phi$ is fully faithful.
\end{lem*}
\begin{proof}
 By \cite[Prop.~2.21 (a)]{DM}, we just have to show the ``if'' part,
 which itself is a slight refinement of \cite[Rmk.~2.29]{DM}. The
 functor $\Phi$ is fully faithful if and only  its restriction to
 $\langle M\rangle$, for every object $M\in \mc{T}$,
 induces an equivalence with $\langle \Phi(M) \rangle$, where $\langle
 M\rangle$ is the full Tannakian subcategory spanned by $M$.  By Tannaka
 duality, this is equivalent to $\Phi^*\colon\pi_1( \langle \Phi(M)
 \rangle,\omega')\to \pi_1( \langle M \rangle,\omega)$ being an
 isomorphism, which by \cite[Prop.~2.21 (b)]{DM} is a closed embedding
 of group schemes of finite type over $\overline{\mathbb{Q}}_p$.
 Chevalley's theorem (\cite[Thm.~1.15]{Bri09},
 \cite[Prop.~3.1.(b)]{Del82}) asserts that
 $\pi_1( \langle \Phi(M) \rangle,\omega')$ is the stabilizer of a line
 $l$ in a finite dimensional representation $V$ of
 $\pi_1( \langle M \rangle,\omega)$.
 Let $N_V$ (resp.\ $L$) be the Tannakian dual of $V$ in $\mc{T}$
 (resp.\ $l$ in $\mc{T}'$). Then the $ \pi_1( \langle \Phi(M)
 \rangle,\omega')$-equivariant inclusion $l\subset V$, induces the
 inclusion $i\colon L\subset \Phi(N_V)$ in $\mc{T}'$.
 By $(\star )$, there is an integer $m>0$, and an object
 $\widetilde{L}\in \mc{T}$, such that  $L^{\otimes
 m}=\Phi(\widetilde{L})$.
 By full faithfulness, there is a uniquely defined inclusion
 $j\colon \widetilde{L}\subset N_V^{\otimes m}$ such that
 $i^{\otimes m}=\Phi(j)\colon L^{\otimes m}\subset \Phi(N_V^{\otimes
 m})=\Phi(N_V)^{\otimes m}$. Thus the stabilizer
 of $l^{\otimes m}\subset V^{\otimes m}$ is $\pi_1( \langle M
 \rangle,\omega)$.
 On the other hand, if $g\in \pi_1( \langle M
 \rangle,\omega)(\overline{\mathbb{Q}}_p)$ acts on
 $l^{\otimes m}\subset V^{\otimes m}$ with eigenvalue $\lambda \in
 \overline{\mathbb{Q}}_p$, it acts on $l\subset V$ with
 eigenvalue $\frac{\lambda}{m} \in \overline{\mathbb{Q}}_p$. This
 implies that $\pi_1( \langle \Phi(M) \rangle,\omega')$ is the
 stabilizer of $l^{\otimes m}$ in $V^{\otimes m}$, thus
 $\Phi^*\big( \pi_1( \langle \Phi(M) \rangle,\omega')\big)=
 \pi_1(\langle M \rangle,\omega)$. This finishes the proof.
\end{proof}

\subsection{}
Let $X_0$ be a smooth {\em geometrically connected} scheme over $k$.
Take a closed point $x_0\in X_0$. The pull-back $i_x^+$ functor induces
a fiber functor
$\mr{Isoc}^\dag(X)\rightarrow\mr{Vec}_{\overline{\mb{Q}}_p}$, which
endows the Tannakian category $\mr{Isoc}^\dag(X)$ with a
neutralization. The fundamental group
is denoted by $\pi_1^{\mr{isoc}}(X,x)$. This is also independent of the
base field. The detailed construction is written in
\cite[2.4.16]{A}. For $M\in\mr{Isoc}^\dag(X)$, we denote by
$\left<M\right>$ the Tannakian
subcategory of $\mr{Isoc}^\dag(X)$ generated by $M$ ({\it i.e.}\ the
full subcategory consisting of subquotient objects of $M^{\otimes
m}\otimes M^{\vee\otimes m'}$ and their direct sums).
Its fundamental group is denoted by $\mr{DGal}(M,x)$, but as the base
point chosen is irrelevant for the further discussion, we just write
$\mr{DGal}(M)$.

\begin{prop} \label{prop:H0}
 Let $M\in\mr{Isoc}^\dag(X)$.
 Assume we have a morphism $f\colon Y_0\rightarrow X_0$ between
 smooth geometrically connected schemes such that
 \begin{quote}
  for any $N\in\left<M\right>$, the induced homomorphism
  \begin{equation}
   \label{fullfaith}
    H^0(X/k,N)\rightarrow H^0(Y/k,f^+N)
  \end{equation}
  is an isomorphism.
 \end{quote}
 Then the homomorphism
 $\mr{DGal}(f^+M)\rightarrow\mr{DGal}(M)$ is an isomorphism.
\end{prop}
\begin{proof}
 By \cite[Prop.~2.21 (b)]{DM}, the homomorphism in question is a closed
 immersion.
 Let $N,N'\in\left<M\right>$. We have
 \begin{equation*}
  \mr{Hom}(N,N')\cong H^0\bigl(X,\shom(N,N')\bigr)\xrightarrow{\sim}
   H^0\bigl(X,f^+\shom(N,N')\bigr)\cong\mr{Hom}\bigl(f^+N,f^+N'\bigr),
 \end{equation*}
 where the second isomorphism holds
 by  assumption
 since $\shom(N,N')\in\left<M\right>$. This implies that the functor
 $f^+\colon\left<M\right>\rightarrow\left<f^+M\right>$ is fully
 faithful.  By Lemma \ref{keytannak}, it suffices to show that for any rank $1$
 object $N$ in $\mr{Isoc}^\dag(X)$, there exists an integer $m>0$ such
 that $N^{\otimes m}$ is trivial. By definition, there exists an integer
 $n\geq 0$ such that $N$ is the pull-back of
 $N_n\in\mr{Isoc}^\dag(X_n)$. Then by \cite[Lem.~6.1]{AbeLL}, there
 exists a rank $1$ isocrystal $L_0\in\mr{Isoc}^\dag(\mr{Spec}(k))$ and
 an integer $m>0$ such that $(N_n\otimes L_n)^{\otimes m}$ is trivial. As $L$ is
 trivial in $\mr{Isoc}^\dag(X) $,  $N^{\otimes m}$ is trivial as well. 
\end{proof}

\begin{rem*}
 Since we use class field theory in the proof, our argument works only
 when the base field is finite.
\end{rem*}

\begin{cor}
 \label{Lefarith}
 Let $f\colon Y_0\rightarrow X_0$ be a morphism between smooth schemes
 over $k$. Assume that $X_0$ is geometrically connected over $k$. 
 Let $M_0\in\mr{Isoc}^\dag(X_0)$ such that for any $N\in\left<M\right>$,
 the homomorphism  (\ref{fullfaith}) is an
 isomorphism. Then the functor $f^+$ induces an equivalence of
 categories $\left<M_0\right>\xrightarrow{\sim}\left<f^+M_0\right>$.
\end{cor}
\begin{proof}
 First, $Y_0$ is geometrically connected, since
 \begin{equation*}
  \dim_{\overline{\mb{Q}}_p}H^0(Y/k,\overline{\mb{Q}}_{p,Y})
   =
   \dim_{\overline{\mb{Q}}_p}H^0(X/k,\overline{\mb{Q}}_{p,X})
   =1.
 \end{equation*}
 Thus, the proposition tells us that
 $f_*\colon\mr{DGal}(f^+M)\xrightarrow{\sim}\mr{DGal}(M)$.

 Now, for any isocrystal $N$ on a smooth geometrically connected scheme
 $Z$ over $k$, we have an isomorphism
 $\mr{DGal}(F_Z^+N)\xrightarrow{\sim}\mr{DGal}(N)$, where $F_Z$ is the
 $s$-th Frobenius endomorphism on $Z$, defined by the functor
 $F_Z^+$. If we have a Frobenius structure $N_0$ on $N$, it induces an
 automorphism $\varphi$ of $\mr{DGal}(N)$. Using this automorphism, the
 category $\left<N_0\right>$
 is equivalent to that of representations $\rho$ of $\mr{DGal}(N)$ such
 that $\rho\circ\varphi=\rho$.

 Let us come back to the situation of the corollary. The Frobenius
 structure on $M_0$ (resp.\ $f^+M_0$) induces an automorphism
 $\varphi_X$ (resp.\ $\varphi_Y$) of the group $\mr{DGal}(M)$ (resp.\
 $\mr{DGal}(f^+M)$), and these automorphisms coincide via the
 isomorphism $f_*$. Thus, we get the claim by the interpretation of the
 categories $\left<M_0\right>$ and $\left<f^+M_0\right>$ above.
\end{proof}

\section{Cohomological Lefschetz theorem} \label{sec:coh}
This section is devoted to showing the existence of $\ell$-adic
companions for tame isocrystals.

\subsection{}
\label{ss:functors}
Let $X$ be a smooth  variety defined over a finite field $k$.
Some of the results of this section hold for objects both in
$\mr{Isoc}^\dag(X_n)$ and in $\mr{Isoc}^\dag(X)$.
For this reason,  we drop the subscripts in the notation.

\medskip
Let $Z\subset X$ be a closed subscheme, we denote by
$j\colon X\setminus Z\hookrightarrow X$ the open immersion, and $i\colon
Z\hookrightarrow X$ the closed immersion.
We introduce the following four functors from $D^{\mr{b}}_{\mr{hol}}(X)$
to itself:
\begin{equation*}
 (+Z):=j_+\circ j^+,\quad (!Z):=j_!\circ j^+,\quad
  \mb{R}\underline{\Gamma}_Z^+:=i_+\circ i^+,\quad
  \mb{R}\underline{\Gamma}_Z^!:=i_+\circ i^!.
\end{equation*}

\begin{prop}
 \label{commutation}
 Let $X$ be a smooth variety, and $Z$ be a smooth divisor.
 Let $C$ be a smooth curve, intersecting with $Z$ transversally. Let
 $M$ be in $D^{\mr{b}}_{\mr{hol}}(X)$ such that $M|_{X\setminus Z}$ is
 in $\mr{Isoc}^\dag(X\setminus Z)$ and is tame along $Z$ with
 nilpotent residues. Then we have a canonical isomorphism in
 $D^{\mr{b}}_{\mr{hol}}(X)$:
 \begin{equation*}
  (!C)(+Z)(M)\xrightarrow{\sim}(+Z)(!C)(M).
 \end{equation*}
\end{prop}
\begin{proof}
 It suffices to show the dual statement
 \begin{equation*}
  \rho_{(X,Z,C)}(M)\colon
   (!Z)(+C)(M)\xrightarrow{\sim}(+C)(!Z)(M).
 \end{equation*}
 We argue by  induction on the dimension of $X$. Assume
 that the statement is known for $\dim(X)\leq d$. We show the lemma for
 $\dim(X)=d+1$. Since the claim is local, we may assume that there
 exists a system of local coordinates $\{t_0,t_1,\dots,t_d\}$ such that
 $Z=(t_0)$, $C=V(t_1,\dots,t_d)$. We put $D:=V(t_1)$.
 To simplify the notation, we denote the boundaries by $Z_C:=Z\cap C$,
 $Z_D:=Z\cap D$. Moreover, we introduce notations of morphisms as
 follows:
 \begin{equation*}
  \xymatrix@C=40pt{
   C\ar[r]^-{i_{C,D}}\ar@/_10pt/[rd]_-{i_C}&
   D\ar[d]^{i_D}\\
  &X.
  }
 \end{equation*}
 First, let us show that
 $$\alpha\colon(!Z_D)\circ i^!_D(M)\xrightarrow{}
 i^!_D\circ (!Z)(M)$$
 is an isomorphism.
 The exact triangle
 $(!D)\rightarrow\mr{id}\rightarrow
 \mb{R}\underline{\Gamma}^+_D\rightarrow$
 induces the following commutative diagram of exact triangles in
 $D^{\mr{b}}_{\mr{hol}}(D)$:
 \begin{equation*}
  \xymatrix{
   (!Z_D)\circ i^!_D\circ(!D)(M)
   \ar[r]\ar[d]_{\alpha_2}&
   (!Z_D)\circ i^!_D(M)
   \ar[r]\ar[d]^{\alpha}&
   (!Z_D)\circ i^!_D\circ
   \mb{R}\underline{\Gamma}^+_D(M)
   \ar[r]^-{+}\ar[d]^{\alpha_1}&\\
  i_D^!\circ (!Z)\circ(!D)(M)\ar[r]&
   i^!_D\circ (!Z)(M)\ar[r]&
   i_D^!\circ (!Z)\circ\mb{R}\underline{\Gamma}^+_D(M)
   \ar[r]^-{+}&
   }
 \end{equation*}
 We claim that $\alpha_1$ is an isomorphism. Indeed, we have
 \begin{align*}
  (!Z_D)\circ i_D^!\circ\mb{R}\underline{\Gamma}^+_D&=
  (!Z_D)\circ i_D^!\circ i_{D+}\circ i_D^+
  \cong
  (!Z_D)\circ i_D^+,\\
  i_D^!\circ (!Z)\circ\mb{R}\underline{\Gamma}^+_D&=
  i_D^!\circ (!Z)\circ i_{D+}\circ i^+_D\cong
  i_D^!\circ i_{D+}\circ (!Z_D)\circ i^+_D\cong
  (!Z_D)\circ i^+_D.
 \end{align*}
 Here we use the isomorphism $i_D^! \circ i_{D+}\cong\mr{id}$ twice, and
 $i_{D!}\cong i_{D+}$ in the second line. Since $\alpha_1$ is the
 identity on $D\setminus Z_D$, the claim is proven.

\medskip

 Thus, it remains to show that $\alpha_2$ is an isomorphism. It is
 obvious that $\alpha_2|_{D\setminus Z_D}$ is an isomorphism. This
 implies that it is enough to check
 \begin{equation*}
  \mb{R}\underline{\Gamma}^+_{Z_D}
   \circ i^!_D \circ (!Z) \circ (!D)(M)
   \cong
   \mb{R}\underline{\Gamma}^+_{Z_D}
   \circ i^!_D \circ (!D\cup Z)(M)=0.
 \end{equation*}
 Since $M|_{X\setminus Z}$ is assumed to be tame with nilpotent
 residues, we use \cite[(3.4.12.1)]{AC} to conclude.

\medskip

 We now complete the proof.
 The exact triangle
 $\mb{R}\underline{\Gamma}^!_C\rightarrow\mr{id}\rightarrow(+C)
 \xrightarrow{+1}$ induces the following commutative diagram of exact
 triangles:
 \begin{equation*}
   \xymatrix{
   (!Z)\circ\mb{R}\underline{\Gamma}^!_C(M)
   \ar[r]\ar[d]_{\beta}&
   (!Z)(M)\ar[r]\ar@{=}[d]&
   (!Z)\circ(+C)(M)\ar[r]^-{+}\ar[d]^{\rho_{(X,Z,C)}(M)}&\\
  \mb{R}\underline{\Gamma}^!_C\circ(!Z)(M)\ar[r]&
   (!Z)(M)\ar[r]&
   (+C)\circ(!Z)(M)\ar[r]^-{+}&.
   }
 \end{equation*}
 This implies that  $\rho_{(X,Z,C)}$ is an isomorphism if and
 only if $\beta$ is an isomorphism.
 We have the following commutative diagram:
 \begin{equation*}
  \xymatrix@C=50pt{
   i_{C+}\circ(!Z_C)\circ i^!_C(M)
   \ar[r]^-{}\ar@{-}[d]_{\sim}&
   i_{C+}\circ i^!_C\circ (!Z)(M)
   \ar@{=}[d]\\
  (!Z)\circ\mb{R}\underline{\Gamma}^!_C(M)
   \ar[r]_-{\beta}&
   \mb{R}\underline{\Gamma}^!_C\circ(!Z)(M).
   }
 \end{equation*}
 Since $i_C$ is a closed immersion, $\beta$ is an isomorphism if and
 only if
 \begin{equation*}
  \rho'_{(X,Z,C)}(M)\colon (!Z_C)\circ i_C^!(M)\rightarrow
 i_C^!\circ(!Z)(M)
 \end{equation*}
 is an isomorphism. Namely,
 \begin{equation}
  \label{localsequ}
  \tag{$\star$}
  \mbox{$\rho_{(X,Z,C)}(M)$ is an isomorphism}
   \Leftrightarrow
  \mbox{$\rho'_{(X,Z,C)}(M)$ is an isomorphism}.
 \end{equation}
 The homomorphism $\rho'_{(X,Z,C)}(M)$ can be computed as follows:
 \begin{align*}
  (!Z_C)\circ i_C^!(M)
  &\cong
  (!Z_C)\circ i^!_{C,D}\circ i^!_{D}(M)
  \xrightarrow[\rho']{\sim}
  i^!_{C,D}\circ (!Z_D)\circ i^!_D(M)\\
  &\xrightarrow[\alpha]{\sim}
  i^!_{C,D}\circ i^!_D\circ (!Z)(M)
  \cong
  i_C^!\circ (!Z)(M),
 \end{align*}
 where $\rho':=\rho'_{(D,Z_D, C)}\bigl(i_D^!(M)\bigr)$, which is an
 isomorphism by the induction hypothesis  (\ref{localsequ}) applied to
 $(D,Z_D,C)$.
\end{proof}

\begin{lem}
 \label{vanishingh1}
 Let $X$ be a projective variety of dimension $d\geq 2$,  let $C$ be a
 curve which is a complete intersection of ample divisors.
 Then, for any $M\in D^{\mr{b}}_{\mr{hol}}(X)$ such that $\H^i(M)=0$
 (cf.\ \ref{arithdmodrev} for $\H^*$) for $i<d$, one has
 $H^n\bigl(X,(!C)(M)\bigr)=0$ for $n=0,1$. 
\end{lem}
\begin{proof}
 We use induction on the dimension of $X$. When $d=2$, the structural
 morphism $\epsilon\colon X\setminus C\rightarrow\mr{Spec}(k)$ is
 affine. By \cite[Prop.~1.3.13]{AC}, this implies that $\epsilon_+$ is
 left t-exact (with respect to the t-structure introduced in
 \ref{arithdmodrev}). By the vanishing condition on the cohomologies of
 $M$, the lemma follows in this case.
 
 \medskip
 
 Assume $d\geq 3$. Let $H\subset X$ be an ample divisor containing
 $C$. The localization triangle for $H$ induces the exact triangle
 exact triangle
 \begin{equation*}
  (!C)\circ(!H)(M)\rightarrow (!C)(M)\rightarrow
   (!C)\circ\mb{R}\underline{\Gamma}^+_H(M)\xrightarrow{+1}
 \end{equation*}
 Using \cite[Prop.~1.3.13]{AC} again and by the assumption on $M$,
 $\H^i((!H)M)=\H^i(M)=0$ for $i<d$. Using the localization sequence,
 $\H^i\mb{R}\underline{\Gamma}^+_H(M)=0$ for $i<d-1$, thus $\H^i(i^+_H
 M)=0$ for $i<d-1$. This implies that
 \begin{equation*}
  H^i\bigl(X,(!C)\circ\mb{R}\underline{\Gamma}^+_H(M)\bigr)\cong
  H^i\bigl(H,(!C)(i^+_H M)\bigr)=0
 \end{equation*}
 for $i=0,1$ by induction hypothesis.
 Moreover, we have
 \begin{equation*}
  H^i\bigl(X,(!C)\circ(!H)(M)\bigr)\xrightarrow{\sim}
     H^i\bigl(X,(!H)(M)\bigr)=0
 \end{equation*}
 for $i=0,1$, where the first isomorphism holds since $C\subset H$, and
 the second since $X\setminus H$ is affine (cf.\ {\it ibid.}).
 This finishes the proof.
 \end{proof}

\begin{cor}
 \label{finalcor}
 Let $X$ be a smooth projective variety, $Z$ be a simple normal
 crossing divisor. Let $C$ be a smooth curve which is a complete
 intersection of ample divisors in good position with respect to $Z$,
 and $i_C: C\setminus Z_C\rightarrow X\setminus Z$ (where $Z_C:=C\cap Z$
 as before) be the closed embedding.
 Then, for any tame isocrystal $M$ with nilpotent residues of
 $\mr{Isoc}^\dag(X \setminus Z)$,
 the homomorphism $H^0(X\setminus Z,M)\rightarrow H^0(C\setminus
 Z_C,i_C^+M)$ is an isomorphism.
\end{cor}
\begin{proof}
 Consider the long exact sequence
 \begin{equation*}
  H^0\bigl(X\setminus Z, !(C\setminus Z_C)(M)\bigr)
   \rightarrow
   H^0(X\setminus Z,M)\xrightarrow{i^*}
   H^0(C\setminus Z_C,i^+_CM)\rightarrow
   H^1\bigl(X\setminus Z,!(C\setminus Z_C)(M)\bigr).
 \end{equation*}
 In order to prove that $i^*$ is an isomorphism, it is sufficient to
 prove that the left and the right terms are both zero.
 Now, let $j\colon X\setminus Z\hookrightarrow X$ be the open
 immersion. We have
 \begin{equation*}
  H^i\bigl(X\setminus Z,!(C\setminus Z_C)(M)\bigr)\cong
   H^i\bigl(X,(+Z)\circ(!C)(j_+M)\bigr)\cong
   H^i\bigl(X,(!C)\circ (+Z)(j_+M)\bigr),
 \end{equation*}
 where the last isomorphism holds by Proposition \ref{commutation}.
 To finish the proof, we note that $j$ is an affine immersion, which
 implies that $\H^i(j_+M)=0$ for $i\neq d$ by \cite[Prop.~1.3.13]{AC}.
 This enables us to apply Lemma \ref{vanishingh1}.
\end{proof}

\subsection{}
\label{lefsthmtame}
We now consider overconvergent $F$-isocrystals, thus reintroduce
subscripts.

\begin{thm*}
 Let $X_0$ be a smooth projective variety over $k$, $Z_0$ be a
 simple normal crossings divisor, and $x_0$ be a closed point of
 $X_0\setminus Z_0$. Let $C_0$ be  a smooth curve in $X_0$
 passing through $x_0$, which is a complete intersection of ample
 divisors in good position with respect to $Z_0$,
 and let $i_C: C_0\setminus Z_{C,0}\rightarrow X_0\setminus Z_0$
 (where $Z_{C,0}:=C_0\cap Z_0$) be the closed immersion.
 Then, for any irreducible $M_0\in\mr{Isoc}^\dag(X_0\setminus Z_0)$, which is  tame
 with nilpotent residues along the boundary, $i^+_CM_0$ is irreducible.
\end{thm*}
\begin{proof}
 We may assume $X_0$ is connected by arguing componentwise.
 We may further assume that $X_0$ is geometrically connected.
 Then, the theorem follows from Corollary~\ref{finalcor}, combined with
 Corollary~\ref{Lefarith}.
\end{proof}

\begin{rem*}
 The existence of the curve $C_0$  follows from
 \cite[Thm.~1.3]{Poo}.
\end{rem*}

\begin{thm} \label{thm:pure}
 Let $X_0$ be a scheme of finite type over $k$, then any object in
 $D^{\mr{b}}_{\mr{hol}}(X_0)$ is $\iota$-mixed.
\end{thm}
\begin{proof}
 By  a  d\'{e}vissage  argument, it suffices to show the theorem when
 $X_0$ is smooth and for objects $M_0$ in $\mr{Isoc}^\dag(X_0)$.

 We first assume that $X_0$ admits a smooth compactification with
 a simple normal crossings boundary divisor, and that  $M_0$ is tame.
 To check that $M_0$ is $\iota$-mixed, it suffices to check that
 any constituent is $\iota$-pure, thus we may assume that $M_0$ is
 irreducible. Twisting by a character, we may further assume that the
 determinant of $M_0$ is of finite order (\cite[Thm.~6.1]{AbeLL}).
 Let  $x_0\in X_0$ be a closed point, 
 $i\colon C_0\hookrightarrow X_0$ be a smooth curve passing through
 $x_0$ as in Remark~\ref{lefsthmtame}. Since $i^+M_0$ is irreducible by
 Theorem~\ref{lefsthmtame}, it is $\iota$-pure of weight $0$ by
 \cite[Thm.~4.2.2]{A}. Thus, $M_0$ is $\iota$-pure of weight $0$.

 We now treat the general case. Let $h\colon X'_0\rightarrow X_0$  be  a
 semistable reduction with respect to $M_0$
 (\cite[Thm.~2.4.4]{SS4}). Then $h^+M_0$ is $\iota$-mixed. As
 $M_0$ is a direct factor of $h_+h^+M_0$ by the trace formalism
 \cite[Thm..~1.5.1]{A}, we conclude that $M_0$ is $\iota$-mixed by
 \cite[Thm.~4.2.3]{AC}.
\end{proof}

\begin{prop}
 \label{comptame}
 Let $X_0$ be a smooth  projective variety over $k$, $Z_0$ be a simple
 normal crossing divisor, and $M_0$ be an isocrystal 
 in ${\rm Isoc}^\dag(X_0\setminus Z_0)$,
 which is tame along $Z_0$ with nilpotent residues.
 Then  $\ell$-adic companions of $M_0$ exist, and they are tame along
 $Z_0$.
\end{prop}
\begin{proof}
 Let us fix $\sigma\colon\overline{\mb{Q}}_p\rightarrow
 \overline{\mb{Q}}_{\ell}$.
 Arguing componentwise, we may assume that $X_0$ is connected. 
 Twisting by a character, we may assume that the determinant of $M_0$ is
 of finite order. Let us show that $M_0$ is algebraic and $\sigma$-unit
 root. Let $x_0\in X_0\setminus Z_0$ be a closed point, and $C_0$ be
 as in Theorem \ref{lefsthmtame} (see Remark \ref{lefsthmtame}). Using
 the notation of {\it ibid.}, $i^+_CM_0$ is irreducible with
 determinant of finite order. This implies that $i^+_CM_0$
 is algebraic and $\sigma$-unit-root by \cite[Thm.~4.2.2]{A},
 thus the claim follows.

 Now, we wish to apply Drinfeld's theorem~\ref{thm:drin_tame} to
 construct the companions.
 Deligne's theorem \ref{Delthmrec} (ii)
 shows that there exists a number field $E$ such that $f_x(M_0,t)\in
 E[t]$ for any finite extension $k'$ of $k$ and $x\in X_0(k')$. Let
 $\lambda$ be the place of $\sigma(E)$
 over $\ell$ corresponding to $\sigma$.
 Put $f_x(t):=f_x(M_0,t)$, and let us show that this collection
 of functions satisfies the assumptions of Theorem \ref{thm:drin_tame}.
 Since $M_0$ is $\sigma$-unit-root, for any smooth curve $C_0$ and a
 morphism $\varphi\colon C_0\rightarrow X_0$, we have a $\sigma$-companion
 ${}_\ell(\varphi^+M_0)$ of $\varphi^+M_0$ which is a
 {\em lisse \'{e}tale $\overline{\mb{Q}}_\ell$-sheaf}
 by \cite[Thm.~4.2.2]{A}. By using \cite[Lem.~2.7]{Dr} and
 \cite[\S2.3]{Dr}, there is a finite extension $F$ of
 $\sigma(E)_\lambda$ such that the monodromy of ${}_\ell(\varphi^+M_0)$
 is in $\mr{GL}(r,F)$
 for any $C_0$ and $\varphi$, thus  the assumption of
 Theorem~\ref{thm:drin_tame} (i) is satisfied.

 Let us check (ii). We put $X'_0:=X_0$.
 Take a smooth curve $\varphi\colon C_0\to X_0$, then since $M_0$ is
 assumed to be tame, the pull-back $\varphi^+M_0$ is tame.
 This implies that a companion ${}_\ell(\varphi^+M_0)$ is tame as well
 by the same argument as \cite[Lem.~2.3]{De}.
 (Alternatively, we may also argue that since the local epsilon factors
 coincide by Langlands correspondence, and since local epsilon factors
 detect the irregularity, the irregularity and Swan conductor coincide
 at each point.)
 In conclusion, the assumption of 
 Theorem~\ref{thm:drin_tame} (ii) is satisfied
 as well, and we may apply Drinfeld's theorem to construct the desired
 companion.
\end{proof}

\section{Wildly ramified case}
In this section, we show the Lefschetz type theorem for 
isocrystals by reduction  to the tame case. We keep the same notations
as in the previous section,
notably $X_0$, $k\subset k_n\subset\overline{k}$, $X_n$, $X$.
If $M_0\in \mr{Isoc}^\dag(X_0)$, we denote by $M^{\mr{ss}}_0$ the
semisimplification in $\mr{Isoc}^\dag(X_0)$, and likewise for
$M\in\mr{Isoc}^\dag(X)$ and $M^{\mr{ss}}$.

\subsection{}
\label{weilcons}
First, we recall the following the well-known consequence of  the Weil
conjectures (see \cite[Cor.~VI.3]{Laf} and \cite[Thm.~3.4.1 (iii)]{weil2}
for an $\ell$-adic counterpart of the theorem).

\begin{thm*}[{\normalfont\cite[Prop.~4.3.3]{A}, \cite[Thm.~4.3.1]{AC}}]
 (i) Let $X_0$ be a geometrically connected smooth scheme over $k$.
 Let $M_0, M'_0 \in {\rm Isoc}^\dag(X_0)$ be $\iota$-pure
 $F$-isocrystals on $X_0$. Assume that
 $M'_0$ is irreducible. Then, the multiplicity of $M'_0$ in $M_0$,
 in other words $\dim  \mr{Hom}(M'_0,M_0^{\mr{ss}})$, is equal to the
 order of pole of $L(X,M_0\otimes M'^{\vee}_0,t)$ at
 $t=q^{-\dim(X_0)}$.

 (ii) Let $M_0\in {\rm Isoc}^\dag(X_0)$ be $\iota$-pure. Then $M$ is
 semisimple.
\end{thm*}

\begin{lem}
 \label{coinH0}
 Let $X_0$ be a geometrically connected smooth scheme over $k$. Let
 $M_0$ be an $\iota$-pure isocrystal,
 and ${}_\ell M_0$ be an $\ell$-adic companion. Then
 \begin{equation*}
  \dim H^0(X,M)=\dim H^0(X,{}_{\ell}M).
 \end{equation*}
\end{lem}
\begin{proof}
 There exists an integer $n\geq a$ such that the number of constituents
 of $M_n$ and $M$ coincide. One has
 \begin{align*}
  \dim H^0(X,M)&=\dim H^0(X,M^{\mr{ss}})=
  \dim \mr{Hom}(\overline{\mb{Q}}_p,M^{\mr{ss}})\\
  &=\sum_{s\in\mb{C}}\dim\bigl(\mr{Hom}
  (\overline{\mb{Q}}_p(s),(M_n)^{\mr{ss}})\bigr),
 \end{align*}
 where the first equality holds since $M$ is semisimple by 
 Theorem \ref{weilcons} (ii), the middle one is by definition, and the
 last one since for any isocrystal $N_n$ in
 $\mr{Isoc}^\dag(X_n)$ such that $N\cong\overline{\mb{Q}}_p$,
 there exists $s\in\mb{C}$ such that
 $N_n\cong\overline{\mb{Q}}_p(s)$. Indeed, fix an isomorphism
 $N\cong\overline{\mb{Q}}_p$. With this identification, $N_n$ yields an
 isomorphism
 $\Phi\colon
 F^{n+}\overline{\mb{Q}}_p\xrightarrow{\sim}\overline{\mb{Q}}_p$. Let
 $\mr{can}\colon F^{n+}\overline{\mb{Q}}_p\xrightarrow{\sim}
 \overline{\mb{Q}}_p$ be
 the Frobenius structure of
 $\overline{\mb{Q}}_p\in\mr{Isoc}(X_n)$. Giving $\Phi$ is equivalent to
 giving $\Phi\circ\mr{can}^{-1}(1)$, which uniquely determines $s$ such
 that $N_n\cong\overline{\mb{Q}}_p(s)$.
 
 By Theorem \ref{weilcons} (i), the dimension of
 $\mr{Hom}(\overline{\mb{Q}}_p(s),M^{\mr{ss}}_n)$ is
 equal to the order of pole of $L(X_n,M_n(-s),t)$ at $t=q^{-dn}$, where $d$
 denotes the dimension of $X_0$. 
 The similar result holds for ${}_{\ell}M_{n}$, by increasing $n$
 if needed,
 so the lemma holds since $M_n(s)$ and ${}_{\ell}M_{n}(s)$ have
 the same $L$-function. 
\end{proof}

\subsection{}
Let $X_0$ be a smooth scheme over $k$.
Let $M_n$ (resp.\ ${}_{\ell}M_n$) be in ${\rm Isoc}^\dag(X_n)$
(resp.\ lisse Weil $\ell$-adic sheaf on $X_n$). We say $M_n$ (resp.\
${}_{\ell}M_n$) {\it satisfies} (C) {\it with respect to a finite
\'{e}tale cover} $X'\rightarrow X$ if it satisfies the following
condition:
\begin{quote}
 (*)
 For any smooth curve   $i\colon C_0\rightarrow X_0$
 such that
 \begin{equation*}
  \#\pi_0(C_0\times _{X_0}X')=\#\pi_0(X'),
 \end{equation*} 
 the pull-back homomorphism $H^0(X,M)\rightarrow
 H^0(C,i^+M)$
 (resp.\ $H^0(X,{}_\ell{M})\rightarrow H^0(C,i^*{}_\ell{M})$)
 is an isomorphism, where $C=C_0\times_{X_0}X$.
\end{quote}

\begin{lem}
 \label{companC}
 Let $X_0$ be a geometrically connected smooth scheme, and $M_n$ be an
 $\iota$-pure isocrystal. Let ${}_{\ell}M_{n}$ be an $\ell$-adic
 companion.  Then if ${}_{\ell}M_{n}$ satisfies {\normalfont (C)} with
 respect to $X'\rightarrow X$,  so does  $M_n$.
\end{lem}
\begin{proof}
 Take $C_0$ as in (*). Since $X_0$ is assumed to be geometrically
 connected, $C_0$ is geometrically connected as well. By definition,
 $i^*{}_\ell M$ is an $\ell$-adic companion of $i^+M$, and these are
 $\iota$-pure. Thus, we have
 \begin{equation*}
  \dim H^0(X,M)=\dim H^0(X,{}_\ell M)=\dim H^0(C,i^*{}_\ell M)
  =\dim H^0(C,i^+ M),
 \end{equation*}
 where the first and the last equality hold by Lemma \ref{coinH0}, and
 the middle one by assumption.
\end{proof}

\begin{lem}
 \label{ladicex}
 Assume $X_0$ is smooth and geometrically connected over $k$.
 Let ${}_{\ell}M_{0}$ be a lisse
 Weil $\overline{\mb{Q}}_\ell$-sheaf on
 $X_0$.
 Then there exists a finite \'{e}tale cover $g\colon X'\rightarrow
 X$ such that any ${}_{\ell}N\in\left<{}_{\ell}M\right>$ satisfies
 {\normalfont (C)} with respect to $g$.
\end{lem}
\begin{proof}
 Let $\rho\colon\pi_1(X)\rightarrow \mr{GL}(r, \overline{\mb{Q}}_\ell)$
 be the representation corresponding to ${}_{\ell}M$, and set
 $G:=\mr{Im}(\rho)$. The argument of \cite[B.2]{EK} holds also for
 schemes over $\overline{k}$, since only the finiteness of
 $H^1_{\mr{\acute{e}t}}(-,\mb{Z}/\ell)$ is used.
 Thus, there exists a finite Galois cover
 $g\colon X'\rightarrow X$ such that for any profinite group $K$ mapping
 to $\pi_1(X)$, such that the composite  $K\rightarrow \pi_1(X)
 \rightarrow \mr{Gal}(X'/X)$ is surjective,  the composite $K\to
 \pi_1(X)\rightarrow G$ is surjective as well.
 Now, ${}_{\ell}N$ is a representation of $G$, and the
 geometric condition on $C$ asserts that the map
 $\pi_1(C)\rightarrow\mr{Gal}(X'/X)$ is surjective,
 thus any ${}_{\ell}N\in\left<{}_{\ell}M\right>$ satisfies
 (C) with respect to $g$.
\end{proof}

\begin{lem}
 \label{compirr}
Assume $X_0$ is smooth and geometrically connected over $k$.
      Let $N_n$ be an $\iota$-pure isocrystal on $X_n$,
 and ${}_{\ell}N_{n}$ be an
 $\ell$-adic companion.  Then if $N_n$ is geometrically irreducible ({\it
 i.e.}\ $N$ is irreducible), so is ${}_{\ell}N_n$.
\end{lem}
\begin{proof}
 We may assume $N_n\neq0$ and $X_n$ is connected, in which case, $X_n$
 is geometrically irreducible.
 On the other hand,  an $\iota$-pure lisse Weil
 $\overline{\mb{Q}}_\ell$-sheaf $L_n$ on $X_n$ is geometrically
 irreducible if and only if $\mr{End}(L)$ is of dimension $1$. Indeed,
 as $\send(L_n)$ is assumed to be $\iota$-pure, $\send(L)$
 is semisimple (\cite[Thm.~3.4.1 (iii)]{weil2}), thus 
 $\dim\mr{End}(L)=\dim\mr{End}(L^{\mr{ss}})$.
To conclude the proof, we have
 $\dim\mr{End}({}_{\ell}N)=\dim\mr{End}(N)$
 by Lemma~\ref{coinH0}, and the latter is equal
 to $1$ since $N$ is assumed to be irreducible.
\end{proof}

\begin{cor}
 \label{constC}
 Let $X_0/k$ be a connected (but not necessarily geometrically connected)
 scheme.
 Let $M_0$ be an $\iota$-pure isocrystal on $X_0$. Assume
  that for any  $n \in \mathbb{N}_{>0}$,  any  $N_n\in\left<M_n\right>$ has 
 an $\ell$-adic companion.
 Then there exists a finite \'{e}tale cover $g\colon X'\rightarrow X$
 such that any $N\in\left<M\right>$ satisfies {\normalfont (C)} with
 respect to $g$.
\end{cor}
\begin{proof}
 First assume $X_0$ is geometrically connected.
 Let $X'\rightarrow X$ be a finite \'{e}tale cover as in Lemma
 \ref{ladicex} for the $\ell$-adic sheaf ${}_{\ell}M$.
 Our goal is to show that this cover satisfies the required
 condition.
\medskip

 Take $N\in\left<M\right>$.
 Since the category $\left<M\right>$ is semisimple, it is enough to
 check (C) for any irreducible $N$.
 Then we can find $N_n\in\left<M_n\right>$ which induces
 $N$ for some $n\geq0$.
 Take integers $m$, $m'$ such that $N_n$ is a subquotient of
 $M_n^{m,m'}$. By Lemma~\ref{companC}, it remains to show that its
 $\ell$-adic companion $N_{\ell,n}$ is in
 $\left<{}_{\ell}M_n\right>$. By Lemma~\ref{compirr}, we know that
 ${}_{\ell}N_n$ is irreducible as well.
 Since $N_n$ is assumed to be a subquotient
 of $M^{m,m'}_n$, $L(X_n,M^{m,m'}_n\otimes N^{\vee}_n,t)$ has a pole
 at $t=q^{-nd}$ by Theorem~\ref{weilcons} (i). This implies that
 $L(X_n,{}_{\ell}M_n^{m,m'}\otimes{}_{\ell}N_n^{\vee},t)$ has a pole
 at $t=q^{-nd}$ as well.
 It follows again by Theorem~\ref{weilcons} (i) that ${}_{\ell}N_n$ is a
 subquotient of ${}_{\ell}M_n^{m,m'}$, as
 ${}_{\ell}M_n^{m,m'}$ is $\iota$-pure.
\medskip

 Finally, assume $X_0$ is not geometrically connected, and let $k'$ be its
 field of constants. By what we have proven so far, there exists an
 \'{e}tale cover
 $X''\rightarrow X_0\otimes_{k'}\overline{k}$ satisfying
 (C) for any object $N\in\left<M\right>$ with respect to the base $k'$.
 By Lemma~\ref{compdiffbase}, the \'{e}tale cover
 \begin{equation*}
  X':=\coprod_{\sigma\in\mr{Gal}(k'/k)}X''
   \rightarrow
   \coprod_{\sigma\in\mr{Gal}(k'/k)}X_0\otimes_{k'}
   \overline{k}
   \cong X_0\otimes_k\overline{k}\,=X
 \end{equation*}
 fulfills the condition of the corollary. This finishes the proof. 
\end{proof}

\begin{thm}   \label{existcov}
 Let $X_0$ be a smooth geometrically connected scheme over $k$.
 Let $M_0 \in {\rm Isoc}^\dag(X_0)$ be  $\iota$-pure.  There exists a dense 
 open subscheme $U_0\hookrightarrow X_0$  and  a finite \'{e}tale cover
 $g: U'\to U$ such that any $N\in\left<M\right>$  satisfies {\normalfont
 (C)} with respect to $g$.
\end{thm}
\begin{proof}
 Note that if  $N_n$ satisfies (C) with respect to $g$, then the
 restriction of $N_n$ to any dense open subscheme $U_0$ satisfies (C)
 with respect to the restriction of $g$ to $U_0$.
 Indeed, if $Y_0$ is a smooth
 scheme over $k$ and $N$ is an isocrystal in $\mr{Isoc}^\dag(Y)$,
 then for any open dense subscheme $U_0\subset Y_0$, the restriction
 $H^0(Y,N)\rightarrow H^0(U,N)$ is an isomorphism
 (in fact, the restriction homomorphism
 $\pi_1^{\mr{isoc}}(U)\rightarrow\pi_1^{\mr{isoc}}(Y)$ is surjective:
 see the proof of \cite[2.4.19]{A}).
 \medskip
 
 By \cite[Thm.~2.4.4]{SS4}, there exists an alteration $h\colon X''_0\rightarrow X_0$ such that
 $X_0''$ is smooth and admits a smooth compactification such that the
 divisor at infinity has strict normal crossings, and such that
 $h^+(M_0)$ is log-extendable with nilpotent residues.
 Since any object of $\left<h^+M_n\right>$
 is log-extendable with nilpotent residues,
 it possesses an $\ell$-adic companion by Proposition
 \ref{comptame}. Thus, we may take $g\colon X'\rightarrow X''$ which
 satisfies (C) for any object in $\left<h^+M\right>$ by Corollary
 \ref{constC}.
 Now, take $U_0\subset X_0$ so that $h$ is a finite \'{e}tale cover. Then
 $h^+(M|_{U_0})$ satisfies (C) with respect to $g|_{h^{-1}(U_0)}$.
   Namely, replacing $X_0$ by $U_0$
 and $X'_0$, $X''$ by their pull-backs, we are in the following
 situation: we have \'{e}tale covers $X''_0\xrightarrow{h}X_0$,
 $X'\xrightarrow{g}X''$ such that any object in $\left<h^+M\right>$
 satisfies (C) with respect to $g$.
 \medskip

 We check now that any object in
 $\left<M\right>$ satisfies (C) with respect to $h\circ g$. Let $C_0$ be
 a curve such that
 $\#\pi_0(C_0\times_{X_0}X')=\#\pi_0(X')$
 and take $N\in\left<M\right>$. Then
 $h^+N\in\left<h^+M\right>$. We have the following diagram:
 \begin{equation*}
  \xymatrix{
   H^0(X'',h^+N)
   \ar[r]^-{\alpha}\ar@/^10pt/[d]^{\mr{tr}}&
   H^0(C\times_{X}X'',h^+N)\ar@/^10pt/[d]^{\mr{tr}}\\
 \ar[u] H^0(X,N)\ar[r]_-{\alpha'}\ar[u] & H^0(C,N),\ar[u]
   }
 \end{equation*}
 where $\mr{tr}$ denotes the trace map \cite[Thm.~1.5.1]{A}.
 The homomorphism $\alpha$ is an isomorphism by construction. 
 Since the trace is functorial, it makes $\alpha'$ a direct summand of
 $\alpha$. Thus $\alpha'$ is an isomorphism as well.
 This finishes the proof.
\end{proof}

\subsection{}
Finally, the existence of the curve $C_0$ is guaranteed by the following
lemma.

\begin{lem*} \label{lem:curve}
 Let $X'\rightarrow X$ be a connected finite \'{e}tale cover. 
 For any finite set of closed points $x^{(j)}\to X_0$, 
 there is a  smooth curve $C_0\rightarrow X_0$ with a factorization
 $x^{(j)}\to C_0 \to X_0$ such that $C_0\times_{X_0}X'$ is irreducible.
\end{lem*}
\begin{proof}
 There is a finite \'{e}tale cover $X'_n\rightarrow X_n$ over
 $k_n$ which base changes to $X'\to X$ over
 $\overline{k}$. Let $h\colon Y_n\to X_0$ be the Galois hull of
 the composite finite \'etale morphism $X'_n\to X_n\to X_0$.
 Then $Y_n$ is geometrically connected over $k_n$ as well.
 It is enough to solve the problem with $X'\to X$ replaced by
 $Y_n\otimes_{k_n}\overline{k}\to X$.
 The morphism $h$ corresponds to a finite quotient $\pi_1(X_0)\to
 \Gamma$. One then applies the construction of  \cite[Thm.~2.15
 (i)]{Dr}. This finishes the proof.
\end{proof}

\begin{thm}[Lefschetz theorem for isocrystals]
 \label{thm:lefschetz} 
 Let $X_0$ be a smooth scheme over $k$, and
 $M_0 \in {\rm Isoc}^\dag(X_0)$ be irreducible. There is a dense open
 subscheme $U_0\hookrightarrow X_0$ such that for any finite set of
 closed points $x^{(j)}\to U_0$,
 there is a smooth curve $C_0\rightarrow X_0$ with a factorization
 $x^{(j)}\to C_0 \to X_0$  such that the pull-back of $M_0$ to
 $C_0$ is irreducible.
\end{thm}
\begin{proof}
 Arguing componentwise, we may assume $X_0$ to be connected, and
 moreover, geometrically connected by changing $k$ if needed.
 By Theorem~\ref{thm:pure},    $M_0$ is $\iota$-pure.
 By Theorem~\ref{existcov},  there is a dense open subscheme
 $U_0\subset X_0$ and a finite \'etale cover $g\colon U'\to U$
 such that any $N\in \langle M \rangle$ satisfies (C) with respect to
 $g$.  This is the open $U_0$ of the theorem. 
 One takes  $C^{U_0}_0$ as in
 Lemma~\ref{lem:curve}, so with factorization $x^{(j)}\to C^{U_0}_0\to
 U_0$  for all $j$, and such that $C^{U_0}_0\times_{U_0} U'$  is
 irreducible. Let  $\Gamma$ be the Zariski
 closure of the image of $C^{U_0}_0$ in $X_0$, $\Gamma ' \to \Gamma$ be
 its normalization, $C_0$ be the normalization of $\Gamma'$ in the field
 of functions of $C^{U_0}_0$.  The morphism $C_0\to X_0$ fulfills the
 required properties by considering Corollary~\ref{Lefarith}.
\end{proof}

\section{Remarks and  applications} \label{sec:appl}
\begin{lem} \label{lem:E}
 Let $X_0$ be a smooth scheme over $k$, and $M_0\in {\rm
 Isoc}^\dag(X_0)$. Assume that $M_0$ is algebraic. Then there is a
 number field $E\subset \overline{\mathbb{Q}}_p$ such that
 $f_{x}(M_0,t)\in E[t]$ for any finite extension $k'$ and $x\in
 X_0(k')$.
\end{lem}
\begin{proof}
 We argue by induction on the dimension of $X_0$. The curve case has
 already been treated. We assume the lemma is known for smooth schemes
 of  dimension  less than $\dim(X_0)$.
 By \cite[Thm.~2.4.4]{SS4}, there exists an alteration $h\colon
 X'_0\rightarrow X_0$ such that
 $X_0'$ is smooth and admits a smooth compactification such that the
 divisor at infinity has strict normal crossings, and such that
 $h^+(M_0)$ is log-extendable. Let $U_0\subset X_0$ be a dense open
 subscheme such that $h|_{h^{-1}(U_0)}$ is finite \'{e}tale. Using
 Theorem~\ref{Delthmrec} (ii), there exists a number field $E_U$ such
 that $f_x(M_0,t)\in E_U[t]$ for any $x\in U_{0}(k')$.
 Now, there exists a finite stratification $\{X_{i,0}\}_{i\in I}$ of
 $X_0\setminus U_0$ by smooth schemes. Since algebraicity is an absolute
 notion, the restriction $M_0|_{X_{i,0}}$ is algebraic as well. Thus, by
 induction hypothesis, there exists a number field $E_i$ such that
 $f_x(M_0,t)\in E_i[t]$ for any $x\in X_{i,0}(k')$. Take $E$
 to be a number field which contain $E_U$ and $E_i$ for $i\in I$.
 Then $E$ is a desired number field.
\end{proof}

We now formulate the existence of $\ell$-adic companions in
general. This has been proven by Kedlaya in \cite[Thm.~5.3]{K2}.
However,  two additional properties follow from our method: 
 $\ell$-adic companions  of irreducible overconvergent isocrystals with
 finite determinant are irreducible, and they are $\ell$-adic \'{e}tale
 sheaves, not only Weil sheaves.

\begin{thm} \label{thm:comp}
 Let $X_0$ be a smooth  geometrically connected scheme over $k$, and
 $M_0\in {\rm Isoc}^\dag(X_0)$ be irreducible with finite
 determinant. Then $\ell$-adic companions exist and they are irreducible
 lisse \'{e}tale $\overline{\mb{Q}}_\ell$-sheaves.
\end{thm}
\begin{proof}
 Using Theorem \ref{thm:lefschetz}, there is a smooth curve
 $\varphi\colon C_0\rightarrow X_0$ such that
 $\varphi^+(M_0)$ is irreducible. This implies that $\varphi^+(M_0)$ is
 irreducible with finite determinant, so it is algebraic and
 $\sigma$-unit-root by \cite[Thm.~4.2.2]{A}.
 Thus there exists a closed point $x$ such that
 $f_x(M_0,t)$ is algebraic and any root is $\ell$-adic unit.
 By Theorem~\ref{Delthmrec} (i), $M_0$ is algebraic and
 $\sigma$-unit-root at any point of $X_0$ and we
 can apply  Lemma~\ref{lem:E} to conclude to the existence of $E$.
 Further, the existence of the companions follows from
 Theorem~\ref{thm:drin_tame} and from the semistable reduction theorem
 as the proof of Proposition \ref{comptame}.
 As for  irreducibility,
 there is a smooth curve $C_0\to X_0$ such that the pull-back
 $M_0|_{C_0}$ of $M_0$ to $C_0$ is irreducible, thus the pull-back of an
 $\ell$-adic companion to $C_0$ is irreducible as well, else a strict
 subobject  would produce a strict subobject of $M_0|_{C_0}$. This
 finishes the proof. 
\end{proof}

\begin{cor} \label{cor:finiteness}
 Let $X_0$ be a smooth scheme over $k$. Let $X_0\hookrightarrow
 \overline{X}_0$ be a normal compactification, $D$ be an effective
 Cartier divisor with support $\overline{X}_0\setminus X_0$,
 $\sigma\colon\overline{\mathbb Q}_p\to\overline{\mathbb Q}_\ell$ is an
 isomorphism for a prime $\ell \neq p$.
 Then there are finitely many isomorphism classes of irreducible
 $M_0\in {\rm Isoc}^\dag(X_0)$ of rank $r$, such that $_\ell M_0^\sigma$
 has ramification bounded by $D$, up to twist by rank $1$ objects in
 ${\rm Isoc}^\dag({\rm Spec}(k))$.
\end{cor}
\begin{proof}
 This is a direct application of Theorem~\ref{thm:comp} and Deligne's
 finiteness theorem  \cite[Thm.~1.1]{EK}, once one knows that the
 correspondence $M_0\rightarrow {}_\ell M_0^\sigma$ is injective, which
 follows from Lemma~\ref{coinH0}.
\end{proof}

We end with two remarks.

\begin{rem} \label{rmk:ss}
 Kedlaya's semistable reduction can be made finite \'etale, at least in
 the case where the base field is finite,
 as asked in \cite[Rmk.~A.1.2]{SS3}.
 Let $X_0$ be a smooth scheme over $k$, and $M_0 \in {\rm
 Isoc}^\dag(X_0)$. Then there exists a finite \'{e}tale cover $g\colon
 X'_0\rightarrow X_0$ such that $g^+(M_0)$ is tamely ramified. Indeed,
 let ${}_{\ell}M_0$ be an $\ell$-adic companion. We take a finite
 \'{e}tale cover $g\colon X'_0\rightarrow X_0$ such that
 $g^*{}_{\ell}M_0$ is tamely ramified on $X'$. Then we claim that
 $g^+M_0$ is tame. Indeed, it suffices to check that for any smooth
 curve $C_0$ and a morphism $i\colon C_0\rightarrow X_0$, the
 restriction $i^+(g^+M_0)$ is tame by Definition \ref{dfntame}.
 Now, $i^+(g^+M_0)$ and $i^*(g^*{}_\ell M_0)$ are companion, and the
 local epsilon factors coincide, thus $i^+(g^+M_0)$ is tame since
 $i^*(g^*{}_\ell M_0)$ is tame by construction.
\end{rem}

\begin{rem} \label{dejong}
 If $M\in {\rm Isoc}^\dag(X)$ is irreducible, it is coming from an
 irreducible $M_n\in {\rm Isoc}^\dag(X_n)$.
 An $\ell$-adic companion ${}_\ell M_n$ has to be irreducible
 as well by Lemma~\ref{compirr}. If in addition  $\pi_1(X)=\{1\}$,
 then ${}_\ell M_n$ comes from $k_{n+1}$. This implies that   $M_n$
 comes from $k_{n+1}$ as well. As  ${\rm Isoc}^\dag(X)$ is semisimple,
 this shows  a (very) weak version of  de Jong's conjecture
 (\cite[Conj.~2.1]{ES}): $\pi_1(X)=\{1\}$ implies that objects in $ {\rm
 Isoc}^\dag(X)$ come from $ {\rm Isoc}^\dag({\rm Spec}(k))$.
\end{rem}

\end{document}